\documentclass{amsart}
\usepackage{amsmath,amssymb,amsrefs}
\usepackage[colorlinks]{hyperref}
\usepackage{booktabs}
\usepackage[all]{xy}

\newtheorem{theorem}{Theorem}[section]
\newtheorem{corollary}[theorem]{Corollary}

\theoremstyle{definition}
\newtheorem{remark}[theorem]{Remark}

\providecommand{\Jac}{\mathop{\rm Jac}\nolimits}

\newcommand{\Adual}{\widehat{A}}
\newcommand{\varphidual}{\widehat{\varphi}}

\newcommand{\QQ}{{\mathbb{Q}}}
\newcommand{\ZZ}{{\mathbb{Z}}}

\newcommand{\mybar}[1]{
  \mathchoice
  {#1\llap{$\overline{\phantom{\displaystyle\rm#1}}$}}
  {#1\llap{$\overline{\phantom{\textstyle\rm#1}}$}}
  {#1\llap{$\overline{\phantom{\scriptstyle\rm#1}}$}}
  {#1\llap{$\overline{\phantom{\scriptscriptstyle\rm#1}}$}}
}  
\renewcommand{\bar}{\mybar}

\newcommand{\Kbar}{\bar{K}}

\begin{document}

\title[Torsion points of large order]
{Genus-$2$ Jacobians with\\ torsion points of large order}

\author{Everett W. Howe} 
\address{Center for Communications Research,
         4320 Westerra Court,
         San Diego, CA 92121, USA}
\email{\href{mailto:however@alumni.caltech.edu}{however@alumni.caltech.edu}}
\urladdr{\url{http://www.alumni.caltech.edu/~however/}}

\date{9 December 2014}


\begin{abstract}
We produce new explicit examples of genus-$2$ curves over the rational numbers
whose Jacobian varieties have rational torsion points of large order.  In
particular, we produce a family of genus-$2$ curves over $\QQ$ whose Jacobians
have a rational point of order $48$, parametrized by a rank-$2$ elliptic curve
over $\QQ$, and we exhibit a single genus-$2$ curve over $\QQ$ whose Jacobian
has a rational point of order $70$, the largest order known.  We also give 
new examples of genus-$2$ Jacobians with rational points of 
order~$27$,~$28$, and~$39$.

Most of our examples are produced by `gluing' two elliptic curves together 
along their $n$-torsion subgroups, where $n$ is either $2$ or~$3$.  The 
$2$-gluing examples arise from techniques developed by the author in joint work
with Lepr\'evost and Poonen 15 years ago. The $3$-gluing examples are made
possible by an algorithm for explicit $3$-gluing over non-algebraically closed
fields recently developed by the author in joint work with Br\"oker, Lauter,
and Stevenhagen.
\end{abstract}

\maketitle

\section{Introduction}
\label{S:intro}

In the late 1970s, 
Mazur~\citelist{\cite{Mazur1977a}\cite{Mazur1977b}\cite{Mazur1978}} determined
the $15$ groups that can appear as the group of rational torsion points on an 
elliptic curve over $\QQ$.  It is not known at present whether or not there are
only finitely many groups (up to isomorphism) that occur as the rational 
torsion subgroups of Jacobians of genus-$2$ curves over~$\QQ$.  Over the past 
$25$ years, researchers have searched for genus-$2$ curves over $\QQ$ whose 
Jacobians have torsion points of large order, and have found or constructed 
examples of curves whose Jacobians have rational torsion points of order $n$,
for $1\le n\le 30$, $32\le n\le36$, and $n\in\{39,40,45,48,60,63\}$
(see \citelist{
\cite{ElkiesWeb}
\cite{Flynn1990}
\cite{HoweLeprevostEtAl2000}
\cite{Leprevost1991a}
\cite{Leprevost1991b}
\cite{Leprevost1992}
\cite{Leprevost1995}
\cite{Leprevost1996}
\cite{Ogawa1994}
\cite{PlatonovPetrunin2012a}
\cite{PlatonovPetrunin2012b}}).
In fact, there are infinite families of genus-$2$ Jacobians over $\QQ$ with
rational torsion points of order $n$ for $1\le n\le 26$ and 
$n\in\{30,32,35,40,45,60\}$; this can be seen from the references cited above,
except for $n\in\{14,16,18,22,26\}$.  For these values of $n$, the existence
of infinite families was proven by Lepr\'evost in an unpublished preprint.
Lepr\'evost obtained Jacobians with torsion points of order $14$, $18$, $22$, 
and $26$ by specializing families of genus-$2$ curves $y^2 = f$ with torsion 
points of order $7$, $9$, $11$, and $13$ so that the polynomial $f$ splits in 
a way that ensures the existence of a rational $2$-torsion point, and he
obtained an infinite family of curves with torsion points of order $16$ by
using the methods of~\cite{Leprevost1992}.

In this paper we present a genus-$2$ curve over $\QQ$ whose Jacobian has a 
rational torsion point of order $70$, the largest order yet discovered. We also
exhibit five genus-$2$ curves whose Jacobians have a rational torsion point of 
order~$28$; previously, only two such curves were
known~\citelist{
\cite{PlatonovPetrunin2012a}*{Theorem~4, p.~288}
\cite{PlatonovZhgunEtAl2013}*{Theorem~3, p.~320}}.  As we explain in
Section~\ref{S70}, we obtain these curves by `gluing' two elliptic curves 
together along their $3$-torsion subgroup, using formulas 
from~\cite{BrokerHoweEtAl2014}*{Appendix}.

We also show that there is an infinite family of genus-$2$ curves over $\QQ$ 
whose Jacobians have a rational torsion point of order $48$, the
second-largest order for which an infinite family of curves is known.  This 
family is produced by using the methods of~\cite{HoweLeprevostEtAl2000}, and
we present it in Section~\ref{S48}.

Finally, by conducting a na\"{\i}ve search of genus-$2$ curves given by 
equations with small coefficients, we find four new examples of genus-$2$
curves whose Jacobians have rational torsion points of large order: three
Jacobians that have a rational point of order~$27$
and one with a rational point of order~$39$. We present these
curves in Section~\ref{S:small}.

\section{Torsion points of order \texorpdfstring{$28$}{28}
         and \texorpdfstring{$70$}{70}}
\label{S70}

Let $E_1$ and $E_2$ be elliptic curves over~$\QQ$, and suppose there is an
isomorphism $\psi\colon E_1[3]\to E_2[3]$ of the $3$-torsion subgroup-schemes
of $E_1$ and $E_2$ that is an anti-isometry with respect to the Weil pairings
on $E_1[3]$ and $E_2[3]$.  Let $G$ be the graph of $\psi$, so that $G$ is a 
maximal isotropic subgroup of $(E_1\times E_2)[3]$ with respect to the
product of the Weil pairings.  Let $A$ be the
abelian surface $(E_1\times E_2)/G$ and let $\varphi\colon E_1\times E_2\to A$
be the natural isogeny. Then there is a commutative diagram
\[
\xymatrix{
E_1\times E_2 \ar[d]^{\varphi} \ar[r]^{3}       &  E_1\times E_2  \\
      A                        \ar[r]^{\lambda} & \Adual\ar[u]_{\varphidual}.
}
\]
Here the top arrow is the multiplication-by-$3$ map and $\Adual$ is the dual
abelian surface of~$A$.  The existence of the isogeny 
$\lambda\colon A\to\Adual$ follows from the fact that $G$ is a maximal
isotropic subgroup of the $3$-torsion of $E_1\times E_2$
(see \cite{Milne1986}*{Proposition~16.8, p.~135}). 
By considering the degrees of the other maps in the diagram, we see that
$\lambda$ is an isomorphism; furthermore, it is a polarization.  Thus, 
$(A,\lambda)$ is a principally-polarized abelian surface, so it is the 
Jacobian of a possibly-singular curve~$C$.  A result of 
Kani~\cite{Kani1997}*{Theorem~3, p.~95} shows that $C$ will be singular if and 
only if $\psi$ is the restriction to $E_1[3]$ of a $2$-isogeny $E_1\to E_2$.

It is straightforward to show that if $C$ is a genus-$2$ curve over $\QQ$
whose polarized Jacobian is $(3,3)$-isogenous over $\QQ$ to a product of two 
elliptic curves $E_1$ and $E_2$ over $\QQ$ with the product polarization, then
$C$ can be obtained from this construction for some anti-isometry 
$\psi\colon E_1[3]\to E_2[3]$; the argument is an easy variant of the proof 
of~\cite{HoweLauter2003}*{Lemma~7, p.~1684}.

Suppose $E_1$ and $E_2$ have rational torsion points of order $N_1$ and~$N_2$,
respectively, and suppose $C$ is a curve whose Jacobian is $(3,3)$-isogenous to
$E_1\times E_2$.  If $N_1$ and $N_2$ are coprime to one another and neither is
divisible by~$3$, then $\Jac C$ has a torsion point of order $N_1N_2$. If $N_1$
and $N_2$ are both divisible by $3$ and if $(N_1/3,N_2/3) = 1$, then $\Jac C$ 
will have a torsion point of order $N_1N_2/3$.

By Mazur's theorem, the possible values of $N_1$ and $N_2$ for elliptic
curves over $\QQ$ are the integers from $1$ to $12$, excluding~$11$. The only 
$(N_1,N_2)$ pairs that will possibly give us new orders of torsion in genus-$2$
Jacobians, or orders for which we have only finitely many examples, are 
$(4,7)$, $(7,8)$, and $(7,10)$.

Algorithm 5.4 of~\cite{BrokerHoweEtAl2014} (which we will refer to as the 
``BHLS $3$-gluing algorithm'') takes as input a pair of elliptic curves $E_1$
and $E_2$ over a base field~$k$, and outputs the list of all of the genus-$2$ 
curves $C$ over $k$ whose Jacobians are $(3,3)$-isogenous (over~$k$) to the
product $E_1 \times E_2$.  As we have noted, such curves $C$ will exist only 
when there is an anti-isometry between the group schemes $E_1[3]$ and $E_2[3]$.
The existence of such an anti-isometry implies that the mod-$3$ Galois
representations attached to $E_1$ and $E_2$ are isomorphic, so before we apply
the BHLS $3$-gluing algorithm to a pair of elliptic curves over $\QQ$, it makes
sense to first check, for several primes $\ell$ of good reduction, that the 
mod-$\ell$ reductions of the two curves have traces of Frobenius that are 
congruent modulo~$3$.

For our pool of candidate elliptic curves we combined two databases of curves 
over $\QQ$: Cremona's database~\cite{CremonaWeb} of all elliptic curves of 
conductor at most $339999$, and the Stein--Watkins 
database~\cite{SteinWatkins2002} of certain elliptic curves of conductor at
most $10^8$ and certain elliptic curves of prime conductor at most $10^{10}$.
For the pairs $(N_1,N_2)$ of interest to us, we went through the combined 
databases and made a list of the curves with $N_1$-torsion points and a list of
the curves with $N_2$-torsion points. Then, for every $E_1$ in the first list
and $E_2$ in the second list, we used the BHLS $3$-gluing algorithm to try to 
glue $E_1$ to $E_2$ along their $3$-torsion subgroups.  Our results follow.

(We also tried using elliptic curves with large torsion subgroups produced by 
specializing the universal elliptic curves with $N$-torsion.  We used the
models for the universal curves given 
in~\cite{HoweLeprevostEtAl2000}*{Table~3, p.~219},
which are based on Kubert's curves~\cite{Kubert1976}*{Table~3, p.~217}, and we 
let the parameter $t$ run through all rational numbers of height at
most~$1000$. This additional pool of elliptic curves did not lead us to any
further examples.)

\subsection{Torsion points of order \texorpdfstring{$4$}{4} 
                                and \texorpdfstring{$7$}{7}}
For $N_1=4$ and $N_2=7$ we found five pairs of curves that we could glue
together.  

\begin{theorem}
\label{T:28}
The Jacobian of each of the following genus-$2$ curves over $\QQ$ has a 
rational torsion point of order $28$\textup{:}
{\allowdisplaybreaks
\begin{align*}
C_{28,1}\colon y^2 + (x^2 + x) y &= x^6 + 3 x^5 + 5 x^4 - 4 x^2 - 10 x + 4\\
C_{28,2}\colon y^2 + (x^2 + x) y &= x^6 + 3 x^5 + 3 x^4 + 13 x^3 - 6 x^2 + 18 x\\
C_{28,3}\colon y^2 + (x^2 + x) y &= 4 x^6 - 2 x^5 + 18 x^4 + 3 x^3 + 13 x^2 + 23 x - 11\\
\begin{aligned}
C_{28,4}\colon y^2 + (x^2 + x) y\\ \ \ 
\end{aligned}     &\begin{aligned}
                    & \!= 28320768 x^6 + 167100960 x^5 + 213557586 x^4 - 35302844 x^3\\[-1ex]
                    &\qquad  + 154134546 x^2 - 155174208 x + 40064896\\
                    \end{aligned}\\
C_{28,5}\colon y^2 + (x^2 + x) y &= 25 x^6 - 455 x^5 + 1675 x^4 + 2494 x^3 + 570 x^2 - 1210 x.
\end{align*}
}
\end{theorem}

\begin{proof}
Table~\ref{T:pairs} lists five pairs of elliptic curves over $\QQ$.  The first
curve in each pair has a rational torsion point of order~$4$, and the second of
each pair has a rational torsion point of order~$7$; torsion points of these 
orders are listed in the fourth column of the table. Applying the BHLS 
$3$-gluing algorithm to the $i$th pair and applying Magma's 
\texttt{ReducedMinimalWeierstrassModel} to the output gives us the genus-$2$ 
curve $C_{28,i}$ listed in the statement of the theorem.
\end{proof}

\begin{table}[ht]
\begin{center}
\begin{tabular}{llll}
\toprule
\#& Cremona label    & Equation                                     & Torsion point\\
\midrule
1 & 182a1    & $y^2 = x^3 +    1122741 x +       310814982$ & $(   891,   44928)$ \\
  & 26b1     & $y^2 = x^3 -       3483 x +          121014$ & $(    27,     216)$ \\[1ex]
2 & 294c1    & $y^2 = x^3 -     255339 x -       109668762$ & $(  1659,   63504)$ \\
  & 294b2    & $y^2 = x^3 -     182763 x +        31201254$ & $(   219,    1296)$ \\[1ex]
3 & 490h1    & $y^2 = x^3 +     146853 x +        34506486$ & $(   315,   10584)$ \\
  & 490k2    & $y^2 = x^3 +    1190133 x +       257487174$ & $(  -117,   10800)$ \\[1ex]
4 & 1518s1   & $y^2 = x^3 -  215054379 x +   2013507848358$ & $(  9579,  912384)$ \\
  & 858k1    & $y^2 = x^3 - 7483623723 x + 249446508217254$ & $( 48459,  769824)$ \\[1ex]
5 & 193930c1 & $y^2 = x^3 - 8212844907 x - 260196865770906$ & $(-51237, 5108400)$ \\
  & 4730k1   & $y^2 = x^3 - 7234611147 x + 236852477159814$ & $( 48843,  118800)$ \\
\bottomrule
\end{tabular}
\end{center}
\vskip1ex
\caption{Pairs of elliptic curves that can be glued along their $3$-torsion.}
\label{T:pairs}
\end{table}
       
\begin{remark}
In fact, Magma's \texttt{TorsionSubgroup} command shows that the Jacobian of
each of these curves has no rational torsion other than that generated by the
point of order~$28$.
\end{remark}

\begin{remark}
As we noted, a result of Kani shows that an anti-isometry 
$\psi\colon E_1[3]\to E_2[3]$ between the $3$-torsion of two elliptic curves
gives rise (via the construction sketched earlier) to a nonsingular curve if
and only if $\psi$ is not the restriction to $E_1[3]$ of a $2$-isogeny from
$E_1$ to~$E_2$.  Therefore, there are two ways for a prime $p$ to be a prime of
bad reduction for a curve over $\QQ$ produced in this way: First, $p$ may be a
prime of bad reduction for $E_1$ or~$E_2$; and second, $p$ may be a prime such 
that the reduction of $\psi$ modulo $p$ is the restriction to $E_{1,p}[3]$ of a 
$2$-isogeny $E_{1,p}\to E_{2,p}$ of the elliptic curves modulo~$p$.  For 
example, the curve $C_{28,4}$ has bad reduction at $439$ for the latter 
reason.  The bad reduction of this curve at $439$ makes the comparatively large
coefficients of its reduced model more understandable.
\end{remark}

\subsection{Torsion points of order \texorpdfstring{$7$}{7} 
                                and \texorpdfstring{$8$}{8}}
For $N_1=7$ and $N_2 = 8$ we found no pairs of elliptic curves
from the Cremona and Stein--Watkins databases that we could
glue together. 

\subsection{Torsion points of order \texorpdfstring{$7$}{7} 
                                and \texorpdfstring{$10$}{10}}
For $N_1=7$ and $N_2 = 10$ we found exactly one pair of elliptic curves
from the Cremona and Stein--Watkins databases that we
could glue together.  

\begin{theorem}
Let $C_{70}$ be the genus-$2$ curve
\[
y^2 + (2 x^3 - 3 x^2 - 41 x + 110) y = x^3 - 51 x^2 + 425 x + 179
\]
over $\QQ$. The Jacobian of $C_{70}$ has a rational torsion point of order~$70$.
\end{theorem}

\begin{proof}
Let $E_1$ and $E_2$ be the following two elliptic curves:
\begin{align*}
&\text{(858k1)} &  y^2 &= x^3 - 7483623723 x + 249446508217254\\
&\text{(66c2)}  &  y^2 &= x^3 + 149013 x + 25726950.
\end{align*}
Then $(48459,769824)$ is a torsion point of order $7$ on~$E_1$, and 
$(147, 7128)$ is a torsion point of order $10$ on~$E_2$. The BHLS $3$-gluing 
algorithm, applied to these curves, gives a single genus-$2$ curve as its 
output.  We obtain the equation for $C_{70}$ given in the theorem by applying
Magma's \texttt{ReducedMinimalWeierstrassModel} function to the curve produced 
by the BHLS algorithm, and then shifting $y$ by polynomials in $x$, and 
shifting $x$ by constants, in order to reduce the size of coefficients.
\end{proof}

\begin{remark}
Let $P_1$ and $P_2$ be the two points at infinity on $C_{70}$.  One can check
that the divisor $(2,17) + (4,23) - P_1 - P_2$ represents a point of 
order~$70$ on the Jacobian of~$C_{70}$. 
\end{remark}

\begin{remark}
Magma's \texttt{TorsionSubgroup} command shows that the Jacobian has no 
rational torsion points other than the multiples of this $70$-torsion point.
\end{remark}
       
\section{Torsion points of order \texorpdfstring{$48$}{48}}
\label{S48}

In this section, we use the techniques of~\cite{HoweLeprevostEtAl2000} to
produce a family of genus-$2$ curves over $\QQ$, parameterized by an elliptic
curve over $\QQ$ of rank~$2$, whose members all have Jacobians with a rational
torsion point of order~$48$.  Prior to this work, only two examples of such
curves had appeared in the 
literature~\cite{PlatonovPetrunin2012b}*{Theorem~3, p.~643}.

First we construct a $1$-parameter family of genus-$2$ curves whose Jacobians
have a rational torsion point of order $24$.  The construction works over any
field whose characteristic is neither $2$, $3$, nor $7$.  For the remainder of
this section we fix such a field $K$, and we let $\Kbar$ denote a separable 
closure of~$K$.

\begin{theorem}
\label{T:24} 
Let $s$ be an element of the field $K$, and set
\begin{align*}
c_4 &= -31 \left(s^4 + 42 s^2 - (32200/93) s - 147\right)\\
c_2 &= 2^8 \big(s^8 + 84 s^6 - (3472/3) s^5 + 1470 s^4 
          - 48608 s^3 + 53508 s^2 + 170128 s + 21609\big)\\
c_0 &= (2^{20} \cdot7/3) s  (s^2 + 7)^3  (s^2 + 63) \\
d   &= s^4 + 42 s^2 + (1736/3) s - 147.
\end{align*}
Suppose  $c_0 d \ne 0$.  Then the equation
\[
d y^2 = x^6 + c_4 x^4 + c_2 x^2 + c_0
\]
defines a nonsingular genus-$2$ curve $C_{24}^s$ over $K$, and the Jacobian of
$C_{24}^s$ has a $K$-rational torsion point of order~$24$.
\end{theorem}

\begin{proof}
Let $g$ be the polynomial $x^3 - 31 x^2 + 256 x$, so that the roots of $g$ in
$\Kbar$ are
\[
\beta_1 = 0, \quad 
\beta_2 = (31 - 3 r)/2, \quad\text{and}\quad 
\beta_3 = (31 + 3 r)/2,
\]
where $r\in\Kbar$ is a square root of~$-7$.  Let $F$ be the elliptic curve
$y^2 = g$ over $K$.  Note that $Q = (32,96)$ is a torsion point of order~$8$ 
on~$F$, and that $4Q = (0,0)$.

Let $s$ be as in the statement of the theorem.  Set
\[ 
a = \frac{-8 (s^4 + 42 s^2 - 147)}{(s^2 + 63)^2}
    \quad\text{and}\quad
b = \frac{16 (s^2 + 7)^3}{(s^2 + 63)^3},
\]
and let $f = x(x^2 + ax + b)$, so that the roots of $f$ in $\Kbar$ are
\[
\alpha_1 = 0, \quad 
\alpha_2 = \frac{4 (s + r)^3 (s - 3 r)}{(s^2 + 63)^2}, \quad\text{and}\quad
\alpha_3 = \frac{4 (s - r)^3 (s + 3 r)}{(s^2 + 63)^2}.
\]
The assumption that $c_0 d\ne 0$ shows that $f$ is separable.  Let $E$ be the
elliptic curve $y^2 = f$, and note that 
\[
P = \left(\frac{4(s^2 + 7)}{s^2 + 63}, \frac{224(s^2 + 7)}{(s^2 + 63)^2}\right)
\]
is a torsion point of order $6$ on $E$, and that $3P = (0,0)$.

Let $\psi\colon E[2]\to F[2]$ be the Galois-module isomorphism that sends 
$(\alpha_i,0)$ to $(\beta_i, 0)$, for $i=1,2,3$.  We check that the condition 
that $d\ne 0$ shows that $\psi$ is not the restriction to $E[2]$ of an 
isomorphism $E\to F$.

Proposition~4 (p.~324) of~\cite{HoweLeprevostEtAl2000} shows how to glue $E$ 
and $F$ together along their $2$-torsion subgroups using $\psi$ to get a 
genus-$2$ curve $C$ whose Jacobian is isomorphic to $(E\times F)/G$, where $G$
is the graph of $\psi$. The formulas 
in~\cite{HoweLeprevostEtAl2000}*{Proposition~4} show that $C$ is given by an 
equation $y^2 = h$, for an explicit sextic polynomial $h\in K[x]$. If we take 
that model for $C$ and replace $x$ and $y$ with
\[
\frac{x}{2^3\cdot (s^2+63)}
\quad\text{and}\quad
\frac{2^{16} \cdot 3^3 \cdot 7^3 \cdot s  (s^2 + 7)^3  d^2 \, y}{(s^2 + 63)^{11}},
\]
respectively, we wind up with the equation for $C_{24}^s$ in the statement of
the theorem.

Let $R$ be the point $(2P,Q)$ on $E\times F$.  The smallest positive integer
$n$ such that $nR$ lies in the kernel $G$ of the natural map
$E\times F\to \Jac C_{24}^s$ is $n = 24$, so the image of $R$ in 
$(\Jac C_{24}^s)(K)$ is a point of order~$24$.
\end{proof}

Next we show that for some values of $s$, the point of order $24$ on
$(\Jac C_{24}^s)(K)$ constructed at the end of the preceding proof is the 
double of a point of order~$48$.

\begin{theorem}
\label{T:48}
Let $s$ be an element of the field $K$ such that the quantity $c_0 d$ from
Theorem~\textup{\ref{T:24}} is nonzero.  Let $D$ be the elliptic curve 
$y^2 = x^3 + 14 x^2 + 196 x$ over~$K$.  Suppose there a nonzero point 
$(z,w)\in D(K)$ such that
\begin{equation}
\label{EQ:s}
s = -21 \ \frac{(z^2 + 196) (z^2 + 56 z + 196) + 32 (z + 14) (z - 14) w}
          {z^4 - 896 z^3 - 24696 z^2 - 175616 z + 38416}.
\end{equation}
Then there is a $K$-rational point of order $48$ on the Jacobian of the curve
$C_{24}^s$ from Theorem~\textup{\ref{T:24}}.
\end{theorem}

\begin{remark}
\label{R:gens}
In the case $K = \QQ$, the Mordell--Weil group of $D$ has rank $2$; it is 
generated by the $2$-torsion point $P_1 = (0,0)$ and the independent points 
$P_2 = (7,-49)$ and $P_3=(16,-104)$ of infinite order. The right-hand side of 
equation~\eqref{EQ:s} is a degree-$16$ function on $D$, so each $s\in\QQ$ 
arises from at most $16$ points of $D(\QQ)$. It follows that there are 
infinitely many genus-$2$ curves over $\QQ$ whose Jacobians have a rational 
torsion point of order $48$.
\end{remark}

\begin{proof}[Proof of Theorem~\textup{\ref{T:48}}]
Let notation be as in the proof of Theorem~\ref{T:24}, and let $A$ be the
abelian surface $E\times F$.  The point $(3P,Q)$ on $A$ maps to a point of 
order $8$ on $J= \Jac C_{24}^s$. Proposition~12 (p.~338) 
of~\cite{HoweLeprevostEtAl2000} gives conditions under which nonrational
points on $A$ will map to rational points of~$J$.  In particular, we can use 
the proposition to determine when there is a rational point of $J$ whose double
is the image of $(3P,Q)$ in $J$; if such a point exists, then $J(K)$ will have
a point of order~$16$, and hence also a point of order $48$.

Proposition~12 of~\cite{HoweLeprevostEtAl2000} deals with the Galois cohomology
groups $H^1(G_K, E[2])$ and $H^1(G_K, F[2])$. As is summarized 
in~\cite{HoweLeprevostEtAl2000}*{\S 3.7}, the group $H^1(G_K, E[2])$ can be 
identified with the kernel of the norm map
\[
L_f^*/L_f^{*2}\to K^*/K^{*2},
\]
where $L_f$ is the $K$-algebra $K[T]/f(T)$. Likewise, $H^1(G_K, F[2])$ can be 
identified with the kernel of the norm
\[
L_g^*/L_g^{*2}\to K^*/K^{*2}.
\]
The polynomials $f$ and $g$ both have linear factors over $K$, and their other
roots involve the square root $r$ of $-7$, so both of these kernels are 
isomorphic to $L^*/L^{*2}$, where $L = K[T]/(T^2 + 7)$.  Let $\rho$ denote the
image of $T$ in $L$. If $K$ does not contain $r$ then there is an isomorphism
$L\cong K(r)$ that sends $\rho$ to~$r$. If $K$ does contain $r$ then there is
an isomorphism $L\cong K\times K$ that sends $\rho$ to $(r,-r)$.

Let $A$ and $B$ be the elements
\[
\frac{4 (s + \rho)^3 (s - 3 \rho)}{(s^2 + 63)^2}
\quad\text{and}\quad
\frac{31 - 3 \rho}{2}
\]
of $L$, corresponding to the roots $\alpha_2$ and $\beta_2$ of $f$ and~$g$. 
Then the map 
\[
\iota\colon E(K)/2E(K) \to L^*/L^{*2}
\]
from~\cite{HoweLeprevostEtAl2000}*{Proposition~12} is given by sending the class
of a finite point $(x,y)$ of $E(K)$ to the class of $x - A$ in $L^*/L^{*2}$, 
provided that $x\ne \alpha_2,\alpha_3$.  Likewise, the map 
\[
\iota'\colon F(K)/2F(K) \to L^*/L^{*2}
\]
sends the class of $(x,y)$ in $F(K)$ to the class of $x-B$, provided that 
$x\ne \beta_2,\beta_3$.

Proposition~12 of~\cite{HoweLeprevostEtAl2000} says that the image of $(3P,Q)$
in $J(K)$ will be the double of a point in $J(K)$ if $\iota(3P) = \iota'(Q)$ 
in $L^*/L^{*2}$. Since $\iota(3P)$ is the class of $0 - A$ mod squares, and 
$\iota'(Q)$ is the class of $32 - B$ mod squares, we would like to check
whether $A(B - 32)$ is a square in $L$.

We compute that $A(B-32) = \ell^2 m$, where
\[
\ell  =  \frac{(1 -\rho)(s + \rho)^2}{s^2 + 63}
\quad\text{and}\quad
m  =  \frac{3(5 - \rho)(s - 3\rho)}{2(s + \rho)},
\]
so $A(B-32)$ is a square if and only if $m$ is a square.  But using the 
expression for $s$ in terms of $z$ and $w$ from the statement of the theorem,
we find that $m = n^2$, where
\[
n = 6 \ \frac{ (7 + 11 \rho) w - 4 z^2 + 784}
          {8 z^2 + 7 (1 - 3 \rho) z + 1568}.
\]          
Therefore, there is a $K$-rational $48$-torsion point on $J$.
\end{proof}

\begin{remark}
One can check that the function on $D$ given by the right-hand side of 
equation~\eqref{EQ:s} is invariant under translation by the $2$-torsion point
$P_1 = (0,0)$.
\end{remark}

\begin{corollary}
The Jacobian of each of the following genus-$2$ curves over $\QQ$ has a 
rational torsion point of order $48$\textup{:}
\begin{align*}
y^2 + (x^2 + x) y &= x^6 - 3 x^5 - 5 x^4 + 14 x^3 +  8 x^2 - 16 x     \\
y^2 + (x^2 + x) y &= x^6 -   x^5 + 5 x^4 - 11 x^3 + 10 x^2 -  6 x + 2 \\
\begin{aligned}
y^2 + (x^2 + x) y\\ \ \ 
\end{aligned}     &\begin{aligned}
                    & \!= 1217 x^6 - 3651 x^5 + 15717859 x^4 - 31429634 x^3 \\[-1ex]
                    &\qquad  + 60403483004 x^2 - 60387768796 x + 80875050306064.\\
                    \end{aligned}
\end{align*}
\end{corollary}

\begin{proof}
These are reduced models of the curves $C_{24}^{-21}$, $C_{24}^3$, and 
$C_{24}^{21/31}$, which come from the values of $s$ obtained as in 
Theorem~\ref{T:48} from the points $P_1$, $P_2$, and $P_3$ of $D(\QQ)$ defined 
in Remark~\ref{R:gens}.
\end{proof}

\begin{remark}
The $1$-parameter family of curves in Theorem~\ref{T:24} was obtained by gluing
a fixed elliptic curve with an $8$-torsion point to a family of elliptic curves 
with a $6$-torsion point. One can also construct a more general family of
examples, as follows.

Given elements $t$ and $u$ of the field $K$, let $E$ be the elliptic curve
$F_6^t$ with a rational $6$-torsion point~$P$ defined 
in~\cite{HoweLeprevostEtAl2000}*{pp.~320--322}, and let $F$ be the curve
$F_8^u$ with a rational $8$-torsion point~$Q$. 
(This requires that $t$ and $u$ 
avoid a certain finite set of values, and for this discussion we tacitly assume 
that all such exceptional values are excluded.)
Then $E$ and $F$ can be glued
together along their $2$-torsion subgroups if their discriminants are equal to
one another, up to squares.  According 
to~\cite{HoweLeprevostEtAl2000}*{Table~6, p.~321}, this will be the case if
there is a nonzero $v\in K$ such that
\[
(9t + 1) / (t + 1) = (2u^2 - 1) v^2.
\]
Since $u$ and $v$ then determine~$t$, this gives us a family of
genus-$2$ curves with a rational point of order $24$ on their Jacobians,
parametrized by an open subset $V$ of the $(u,v)$-plane.  

Following the same argument as in the proof of Theorem~\ref{T:48}, we see 
that the $24$-torsion point on the Jacobian corresponding to a given $(u,v)$
pair will be the double of a rational point of order $48$ if and only if
a certain element of the quadratic algebra $L$ determined by the discriminants
of $E$ and $F$ is a square.  This condition can be rephrased as saying that
a given $(u,v)$ pair gives a curve with a $48$-torsion point on its Jacobian
if and only if $(u,v)$ is the image of a rational point under a map $U\to V$ 
from a surface $U$ to $V$.

The equations we derived for the surface $U$ are lengthier than we would like
to present here.  We looked for curves of small genus on $U$, and were able to
find a number of curves of genus~$1$, but none of genus~$0$. The family given
in Theorem~\ref{T:48} corresponds to the fiber of $U$ over the line $u = 1/3$
in~$V$.  (The variable $s$ in Theorem~\ref{T:24} is then $7v/3$.)
\end{remark}

\section{Curves with small coefficients}
\label{S:small}

In the literature, one finds several examples of genus-$2$ curves over $\QQ$
with torsion points of large order on their Jacobians and with models whose 
defining equations have coefficients of very small height.  Inspired by these
examples, we searched through a number of families of curves with small-height 
coefficients in search of further examples.  We found no new orders of torsion 
points, but we did find some new curves, as well as some small models of curves
already in the literature.  

We had the most success when searching for curves of the form
\[
y^2 + (a_3 x^3 + a_2 x^2 + a_1 x + a_0) y = b_2 x^2 + b_1 x + b_0.
\]
(Note that every genus-$2$ curve with a rational non-Weierstrass point has a
model of this form.) For this family, we let the coefficients run through the
integers from $-10$ through $10$.  (By changing the signs of $x$ and $y$, we 
could assume that $a_3$ was positive and $a_2$ nonnegative.) We limited our 
search to torsion orders for which there is \emph{not} a known infinite family
of curves with torsion points of that order; let us call such orders
\emph{interesting}.  For most curves, we could quickly show that the curve's
Jacobian had no interesting rational torsion by looking at the number of 
points on the Jacobians of the reductions of the curve modulo several small
primes of good reduction. For curves whose reductions did allow for the 
existence of torsion points of interesting order, we used Magma's 
\texttt{TorsionSubgroup} command to compute the actual torsion subgroup of the
Jacobian of the curve.

Table~\ref{T:small} gives the curves we found, the order of the torsion point
of largest order on the Jacobian, and, if applicable, a reference to where the
curve has appeared previously in the literature. In some of the examples, we
give a model where there is an $x^3$ term on the right-hand side of the curve's
equation, because allowing that term reduced the coefficient size or the
number of nonzero coefficients.

\begin{table}[ht]
\begin{center}
\begin{tabular}{cl@{\ $+$\ }r@{\ $=$\ }ll}
\toprule
Order & \multicolumn{3}{l}{Equation} & Reference\\
\midrule
$27$ & $y^2$&$ (6 x^3 + 3 x^2 + 3 x - 2) y $&$                 6 x$     & \cite{Leprevost1995}*{Th\'eor\`eme~1.2.1}\\
     & $y^2$&$ (  x^3         - 2 x + 1) y $&$   x^3$                   & New\\
     & $y^2$&$ (2 x^3 + 3 x^2 - 3 x + 2) y $&$ 6 x^3               + 6$ & New\\
     & $y^2$&$ (6 x^3 + 9 x^2 + 6 x - 1) y $&$       - 3 x^2$           & New\\
$28$ & $y^2$&$ (3 x^3 + 2 x^2       + 1) y $&$       -   x^2 -   x$     & \cite{PlatonovPetrunin2012a}*{Theorem~4}\\
     & $y^2$&$ (2 x^3 - 3 x^2 + 3 x + 4) y $&$                 4 x$     & \cite{PlatonovZhgunEtAl2013}*{Theorem~3}\\
$29$ & $y^2$&$ (2 x^3 - 2 x^2 -   x + 1) y $&$                   x$     & \cite{Leprevost1995}*{Th\'eor\`eme~1.2.1} \\
$33$ & $y^2$&$ (3 x^3 + 9 x^2 +   x + 2) y $&$               - 8 x$     & \cite{PlatonovPetrunin2012b}*{Corollary~1}\\
$34$ & $y^2$&$ (6 x^3 + 5 x^2       - 4) y $&$       - 4 x^2$           & \cite{ElkiesWeb}*{``$N=34$''}\\
$36$ & $y^2$&$ (6 x^3 - 3 x^2 -   x + 2) y $&$ 3 x^3 - 4 x^2 + 2 x$     & \cite{PlatonovPetrunin2012b}*{Theorem~2}\\
     & $y^2$&$ (6 x^3 + 3 x^2 -   x + 2) y $&$         2 x^2 + 2 x$     & \cite{PlatonovZhgunEtAl2013}*{Theorem~3}\\
$39$ & $y^2$&$ (  x^3         + 2 x - 1) y $&$ 3 x^3$                   & \cite{ElkiesWeb}*{``$N=39$''}\\
     & $y^2$&$ (6 x^3 + 6 x^2 - 7 x - 9) y $&$               - 2 x - 2$ & New\\
\bottomrule
\end{tabular}
\end{center}
\vskip1ex
\caption{Examples of genus-$2$ curves with torsion points of large order.}
\label{T:small}
\end{table}

With help from Reinier Br\"oker, we also searched specifically for genus-$2$
curves over $\QQ$ with rational $31$-torsion points on their Jacobian.  We
searched through all curves of the form $y^2 = f$, where $f\in\ZZ[x]$ is a
quintic or sextic with all coefficients bounded in absolute value by $20$.
We also searched through all genus-$2$ curves of the form $y^2 + gy = f$,
where $g$ and $f$ are polynomials in $\ZZ[x]$ of degree at most $3$ and~$6$,
respectively, and with all coefficients bounded in absolute value by $5$. 
We found no examples.



\begin{bibdiv}
\begin{biblist}

\bib{BrokerHoweEtAl2014}{misc}{
   author={Reinier Br\"oker},
   author={Everett W. Howe},
   author={Kristin E. Lauter}, 
   author={Peter Stevenhagen},
   title={Genus-$2$ curves and Jacobians with a given number of points},
   date={2014},
   note={\href{http://arxiv.org/abs/1403.6911}{arXiv:1403.6911 [math.NT]}.
        To appear in LMS J. Comput. Math.},
}

\bib{CremonaWeb}{misc}{
    author= {Cremona, John},
    title={Elliptic curve data},
    note={\url{http://homepages.warwick.ac.uk/staff/J.E.Cremona/ftp/data/INDEX.html}},
    date={retrieved 8 May 2014},
}    

\bib{ElkiesWeb}{misc}{
   author={Elkies, Noam D.},
   title={Curves of genus 2 over Q whose Jacobians are absolutely 
          simple abelian surfaces with torsion points of high order},
   note={\url{http://www.math.harvard.edu/~elkies/g2_tors.html}},
   date={retrieved 1 May 2014},
}   

\bib{Flynn1990}{article}{
   author={Flynn, E. V.},
   title={Large rational torsion on abelian varieties},
   journal={J. Number Theory},
   volume={36},
   date={1990},
   number={3},
   pages={257--265},
   issn={0022-314X},
   doi={\href{http://dx.doi.org/10.1016/0022-314X(90)90089-A}{10.1016/0022-314X(90)90089-A}},
}

\bib{HoweLauter2003}{article}{
   author={Howe, E. W.},
   author={Lauter, K. E.},
   title={Improved upper bounds for the number of points on curves over finite
          fields},
   journal={Ann. Inst. Fourier (Grenoble)},
   volume={53},
   date={2003},
   number={6},
   pages={1677--1737},
   issn={0373-0956},
   doi={\href{http://dx.doi.org/10.5802/aif.1990}{10.5802/aif.1990}},
}

\bib{HoweLeprevostEtAl2000}{article}{
   author={Howe, Everett W.},
   author={Lepr{\'e}vost, Franck},
   author={Poonen, Bjorn},
   title={Large torsion subgroups of split Jacobians of curves of genus two or 
          three},
   journal={Forum Math.},
   volume={12},
   date={2000},
   number={3},
   pages={315--364},
   issn={0933-7741},
   doi={\href{http://dx.doi.org/10.1515/form.2000.008}{10.1515/form.2000.008}},
}

\bib{Kani1997}{article}{
   author={Kani, Ernst},
   title={The number of curves of genus two with elliptic differentials},
   journal={J. Reine Angew. Math.},
   volume={485},
   date={1997},
   pages={93--121},
   issn={0075-4102},
   doi={\href{http://dx.doi.org/10.1515/crll.1997.485.93}{10.1515/crll.1997.485.93}},
}


\bib{Kubert1976}{article}{
   author={Kubert, Daniel Sion},
   title={Universal bounds on the torsion of elliptic curves},
   journal={Proc. London Math. Soc. (3)},
   volume={33},
   date={1976},
   number={2},
   pages={193--237},
   issn={0024-6115},
   doi={\href{http://dx.doi.org/10.1112/plms/s3-33.2.193}{10.1112/plms/s3-33.2.193}},
}


\bib{Leprevost1991a}{article}{
   author={Lepr{\'e}vost, Franck},
   title={Famille de courbes de genre $2$ munies d'une classe de diviseurs 
          rationnels d'ordre $13$},
   journal={C. R. Acad. Sci. Paris S\'er. I Math.},
   volume={313},
   date={1991},
   number={7},
   pages={451--454},
   issn={0764-4442},
   note={\url{http://gallica.bnf.fr/ark:/12148/bpt6k57325582/f455.image}}
}

\bib{Leprevost1991b}{article}{
   author={Lepr{\'e}vost, Franck},
   title={Familles de courbes de genre $2$ munies d'une classe de diviseurs
          rationnels d'ordre $15,\;17,\;19$ ou $21$},
   journal={C. R. Acad. Sci. Paris S\'er. I Math.},
   volume={313},
   date={1991},
   number={11},
   pages={771--774},
   issn={0764-4442},
   note={\url{http://gallica.bnf.fr/ark:/12148/bpt6k57325582/f775.image}}
}


\bib{Leprevost1992}{article}{
   author={Lepr{\'e}vost, Franck},
   title={Torsion sur des familles de courbes de genre $g$},
   journal={Manuscripta Math.},
   volume={75},
   date={1992},
   number={3},
   pages={303--326},
   issn={0025-2611},
   doi={\href{http://dx.doi.org/10.1007/BF02567087}{10.1007/BF02567087}},
}


\bib{Leprevost1995}{article}{
   author={Lepr{\'e}vost, Franck},
   title={Jacobiennes de certaines courbes de genre $2$: torsion et
          simplicit\'e},
   journal={J. Th\'eor. Nombres Bordeaux},
   volume={7},
   date={1995},
   number={1},
   pages={283--306},
   issn={1246-7405},
   note={Les Dix-huiti\`emes Journ\'ees Arithm\'etiques (Bordeaux, 1993).
        \url{http://www.emis.de/journals/JTNB/1995-1/jtnb7-1.html}},

}


\bib{Leprevost1996}{article}{
   author={Lepr{\'e}vost, Franck},
   title={Sur une conjecture sur les points de torsion rationnels des
          jacobiennes de courbes},
   journal={J. Reine Angew. Math.},
   volume={473},
   date={1996},
   pages={59--68},
   issn={0075-4102},
   doi={\href{http://dx.doi.org/10.1515/crll.1995.473.59}{10.1515/crll.1995.473.59}},
}


\bib{Mazur1977a}{article}{
   author={Mazur, B.},
   title={Rational points on modular curves},
   conference={
      title={Modular functions of one variable, V},
      address={Proc. Second Internat. Conf., Univ. Bonn, Bonn},
      date={1976},
   },
   book={
      publisher={Springer, Berlin},
      editor={J.-P. Serre},
      editor={D. B. Zagier},
   },
   date={1977},
   pages={107--148. Lecture Notes in Math., Vol. 601},
   doi={\href{http://dx.doi.org/10.1007/BFb0063947}{10.1007/BFb0063947}},
}

\bib{Mazur1977b}{article}{
   author={Mazur, B.},
   title={Modular curves and the Eisenstein ideal},
   journal={Inst. Hautes \'Etudes Sci. Publ. Math.},
   number={47},
   date={1977b},
   pages={33--186 (1978)},
   issn={0073-8301},
   doi={\href{http://dx.doi.org/10.1007/BF02684339}{10.1007/BF02684339}},
}

\bib{Mazur1978}{article}{
   author={Mazur, B.},
   title={Rational isogenies of prime degree (with an appendix by D.
          Goldfeld)},
   journal={Invent. Math.},
   volume={44},
   date={1978},
   number={2},
   pages={129--162},
   issn={0020-9910},
   doi={\href{http://dx.doi.org/10.1007/BF01390348}{10.1007/BF01390348}},
}


\bib{Milne1986}{article}{
   author={Milne, J. S.},
   title={Abelian varieties},
   conference={
      title={Arithmetic geometry},
      address={Storrs, Conn.},
      date={1984},
   },
   book={
      publisher={Springer, New York},
      editor={Gary Cornell},
      editor={Joseph H. Silverman},
   },
   date={1986},
   pages={103--150},
   doi={\href{http://dx.doi.org/10.1007/978-1-4613-8655-1_5}{10.1007/978-1-4613-8655-1\_5}},
}


\bib{Ogawa1994}{article}{
   author={Ogawa, Hiroyuki},
   title={Curves of genus $2$ with a rational torsion divisor of order $23$},
   journal={Proc. Japan Acad. Ser. A Math. Sci.},
   volume={70},
   date={1994},
   number={9},
   pages={295--298},
   issn={0386-2194},
   note={\url{http://projecteuclid.org/euclid.pja/1195510899}}
}


\bib{PlatonovPetrunin2012a}{article}{
   author={Platonov, V. P.},
   author={Petrunin, M. M.},
   title={New orders of torsion points in Jacobians of curves of genus 2
          over the rational number field},
   journal={Dokl. Math.},
   volume={85},
   date={2012},
   number={2},
   pages={286--288},
   issn={1064-5624},
   doi={\href{http://dx.doi.org/10.1134/S1064562412020330}{10.1134/S1064562412020330}},
   note={Translated from Dokl. Akad. Nauk {\bf 443} (2012), no.~6 664--667},
}


\bib{PlatonovPetrunin2012b}{article}{
   author={Platonov, V. P.},
   author={Petrunin, M. M.},
   title={On the torsion problem in Jacobians of curves of genus 2 over the
          rational number field},
   journal={Dokl. Math.},
   volume={86},
   date={2012},
   number={2},
   pages={642--643},
   issn={1064-5624},
   doi={\href{http://dx.doi.org/10.1134/S1064562412050304}{10.1134/S1064562412050304}},
   note={Translated from Dokl. Akad. Nauk {\bf 446} (2012), no.~3, 263--264},
}

\bib{PlatonovZhgunEtAl2013}{article}{
  author={Platonov, V. P.},
  author={Zhgun, V. S.},
  author={Petrunin, M. M.},
  title={On the simplicity of Jacobians for curves of genus 2 over the
         rational number field containing torsion points of large orders},
  journal={Dokl. Math.},
  volume={87},
  number={3},
  date={2013},
  pages={318--321},
  doi={\href{http://dx.doi.org/10.1134/S1064562413030216}{10.1134/S1064562413030216}},
  note={Translated from Dokl. Akad. Nauk {\bf 450} (2013), no.~4, 385--388},
}

\bib{SteinWatkins2002}{article}{
   author={Stein, William A.},
   author={Watkins, Mark},
   title={A database of elliptic curves---first report},
   conference={
      title={Algorithmic number theory},
      address={Sydney},
      date={2002},
   },
   book={
      series={Lecture Notes in Comput. Sci.},
      volume={2369},
      publisher={Springer, Berlin},
      editor={Claus Fieker},
      editor={David R. Kohel},
   },
   date={2002},
   pages={267--275},
   doi={\href{http://dx.doi.org/10.1007/3-540-45455-1_22}{10.1007/3-540-45455-1\_22}},
}


\end{biblist}
\end{bibdiv}
\end{document}